\documentclass[12pt]{article}

\usepackage{amsmath}
\usepackage{amssymb}

\newcommand{\brk}[1]{\left(#1\right)}          
\newcommand{\Brk}[1]{\left[#1\right]}          
\newcommand{\BRK}[1]{\left\{#1\right\}}        
\newcommand{\Average}[1]{\left<#1\right>}      
\newcommand{\Abs}[1]{\left| #1 \right|}        
\newcommand{\Norm}[1]{\left\| #1 \right\|}     

\newcommand{\pd}[2]{\frac{\partial#1}{\partial#2}}
\newcommand{\deriv}[2]{\frac{d#1}{d#2}}

\newcommand{\bs}[1]{\boldsymbol{#1}}

\newcommand{\grad}{\boldsymbol{\nabla}}
\providecommand{\half}{\frac{1}{2}}

\newcommand{\R}{{\mathbb R}}

\newcommand{\Textand}{\qquad\text{ and }\qquad}

\newcommand{\str}{{\boldsymbol{\sigma}}}
\newcommand{\btau}{{\boldsymbol{\tau}}}
\newcommand{\vel}{{\boldsymbol{u}}}
\newcommand{\vort}{{\boldsymbol{\omega}}}
\newcommand{\gvel}{{\grad\vel}}

\newcommand{\x}{{\boldsymbol{x}}}
\newcommand{\y}{{\boldsymbol{y}}}

\newcommand{\F}{{\boldsymbol{F}}}

\renewcommand{\div}{\operatorname{div}}
\newcommand{\tr}{\operatorname{tr}}
\newcommand{\PV}{\operatorname{(P.V.)}}

\newcommand{\e}{{\varepsilon}}
\newcommand{\Fe}{{\F_\e}}
\newcommand{\stre}{{\str_\e}}
\newcommand{\stret}{{\str_{\e'}}}
\newcommand{\vele}{{\vel_\e}}
\newcommand{\velet}{{\vel_{\e'}}}
\newcommand{\gvele}{{\gvel_\e}}
\newcommand{\gvelet}{{\gvel_{\e'}}}
\newcommand{\Je}{{J_\e}}
\newcommand{\Jet}{{J_{\e'}}}

\newcommand{\Cinf}{C^\infty}

\newcommand{\Hm}{H^m(\R^3)}
\newcommand{\Hk}{H^k(\R^3)}
\newcommand{\Hmp}{H^{m'}(\R^3)}
\newcommand{\Ltwo}{L^2(\R^3)}
\newcommand{\Linf}{L^\infty(\R^3)}

\newcommand{\NormH}[2]{\|#1\|_{#2}}
\newcommand{\NormLinf}[1]{\|#1\|_{L^\infty}}
\newcommand{\NormLq}[2]{\|#1\|_{L^{#2}}}

         \newtheorem{thm}{Theorem}[section]
          \newtheorem{prop}[thm]{Proposition}
          \newtheorem{lem}[thm]{Lemma}

\newenvironment{proof}{{\flushleft \emph{Proof}:}}{$\blacksquare$}

\numberwithin{equation}{section}

\title{A Beale-Kato-Majda breakdown criterion for an Oldroyd-B fluid  in the creeping flow regime}

\author{Raz Kupferman \footnotemark[1]
\and
Claude Mangoubi  \footnotemark[2]
\and
Edriss S. Titi \footnotemark[3]
}

\begin{document}


\maketitle

\renewcommand{\thefootnote}{\fnsymbol{footnote}}
\footnotetext[1]{Institute of Mathematics, The Hebrew University,
Jerusalem 91904 Israel}
\footnotetext[2]{Institute of Mathematics, The Hebrew University,
Jerusalem 91904 Israel and CERMICS, \'Ecole Nationale des Ponts et Chauss\'es,
77455 Marne la Vall\'ee, France}
\footnotetext[3]{Faculty of Mathematics and Computer Science, The Weizmann Institute of Science,
Rehovot 76100 Israel and
Department of Mathematics and Department of Mechanical and Aerospace Engineering, University of California, Irvine CA 92697-3875}
\renewcommand{\thefootnote}{\arabic{footnote}}

\begin{abstract}
We derive a criterion for the breakdown of solutions to the Oldroyd-B model in $\R^3$ in the limit of zero
Reynolds number (creeping flow). If the initial stress field is in the Sobolev space $\Hm$, $m> 5/2$, then
either a unique solution exists within this space indefinitely, or, at the   time where the solution breaks
down, the time integral of the $L^\infty$-norm of the stress tensor must diverge. This result is analogous to the
celebrated Beale-Kato-Majda breakdown criterion for the inviscid Eluer equations of incompressible fluids.
\end{abstract}



\pagestyle{myheadings}
\thispagestyle{plain}
\markboth{R. KUPFERMAN, C. MANGOUBI \& E. TITI}{BEALE-KATO-MAJDA CRITERION FOR AN OLDROYD-B FLUID}

\section{Introduction}
The Oldroyd-B model is a classical model for dilute solutions of polymers suspended in a viscous incompressible
solvent \cite{BAH87}.  Although it suffers, as a physical model derived from microscopic dynamics,
from numerous shortcomings (e.g., polymers are allowed to stretch indefinitely), it is often considered
as a prototypical model for viscoelastic fluids, and has therefore been the focus of both analytical and
numerical studies.

At present, there is no global-in-time existence theory for the
Oldroyd-B model. The notable difference between the Oldroyd-B model
and its Newtonian counterpart, the incompressible Navier-Stokes equations, is that
in the viscoelastic case, global-in-time existence has not even been
established in two dimensions nor in the creeping flow regime, i.e.,
when the momentum equation is a Stokes system. The reason for this
difference can be understood by observing structural similarities
between the inertia-less Oldroyd-B model and the Euler
equations (in three dimensions), or the 2D quasi-geostrophic flow equations
(in two dimensions) \cite{CMT94}.

Since the early 1970s, numerical simulations of the Oldroyd-B model
(as well as other viscoelastic models) have been infested by
stability and accuracy problems that arise at frustratingly low
values of the elasticity parameter (the Weissenberg number)
\cite{Keu86,Keu00}. While some of these difficulties have been
elucidated \cite{FK05}, it is to a large extent still a mystery why
computations break down in the low-Reynolds-high-Weissenberg regime.
As is often the case in such situations, numerical data are by
themselves not sufficient to explain the reasons for this breakdown.
In the absence of a well-posedness theory, it is not even clear in
what spaces solutions have to be sought. Thus, the development of
such a theory is of major importance for both theoretical and
practical purposes.

This situation is analogous to that of incompressible Newtonian fluids at high
Reynolds number, where global-in-time existence has not yet been
established.
For a Newtonian fluid (in three dimensions), however, there is a
prominent observation due to Beale, Kato and Majda (BKM)
\cite{BKM84}, which states a necessary  and sufficient
condition for the breakup of
solutions at finite time. Specifically, the Euler equations of incompressible inviscid fluids
in vorticity formulation take the form
\begin{equation}
\pd{\vort}{t} + (\vel\cdot\grad)\vort = (\gvel)^T\vort,
\label{eq:euler}
\end{equation}
with initial condition $\vort(\cdot,0) = \vort_0$. Here $\vort=\nabla\times\vel$ is
the vorticity and $\vel$ is the velocity field. The BKM theorem
states that if $\vort_0$ belongs to the Sobolev space $\Hm$,
$m>\frac32$, then either there exists a solution
$\vort(\cdot,t)\in\Hm$ for all times, or, if $T^*$ is the maximal
time of existence of a solution in $\Hm$, then
\[
\lim_{t\nearrow T^*} \int_0^t \NormLinf{\vort(\cdot,s)}\,ds = \infty,
\]
and in particular,
\[
\limsup_{t\nearrow T^*}\NormLinf{\vort(\cdot,t)}=\infty.
\]
That is, the breakup of solutions in
any Sobolev norm necessitates the divergence of the $L^\infty$-norm
of the vorticity. The practical implication of this theorem is that
breakdown cannot be attributed, say, to the failure of some high
derivative. The blowup of the vorticity itself, in the
supremum norm, is the signature of any finite-time
breakdown. For another criterion of singularity formation see \cite{CFM96};
for up-to-date surveys see \cite{BT07,Con07}.

The goal of this paper is to establish a similar result for the
three-dimensional Oldroyd-B model, in the zero-Reynolds number
regime. In this regime, a closed equation can be written for the
polymeric stress field $\str= \str(\x,t)$; this
equation is similar to the vorticity equation \eqref{eq:euler}.
We start by establishing the local-in-time existence of solutions in
any Sobolev space $\Hm$, $m> 5/2$. Following then the approach of
BKM, we prove that if the initial stress is in $\Hm$, then either a
solution exists for all time, or, if $T^*$ is the maximum existence
time, then
\[
\lim_{t\nearrow T^*} \int_0^t \NormLinf{\str(\cdot,s)}\,ds = \infty.
\]
This result is independent of the Weissenberg number, and in fact,
holds even if one sets the Weissenberg number to be infinite. From a
theoretical point of view, this breakdown condition  implies that
global-in-time well-posedness hinges on an a-priori bound for the
supremum norm of the stress.

Recent work along these lines comprises a BKM-type analysis by
Chemin and Masmoudi \cite{CM01}. The notable difference between
their analysis and the present work is that they treat the Oldroyd-B
model including inertia, however, their analysis is restricted to
two-dimensional flows. Lin et al. \cite{LLZ05} analyze the inertial
Oldroyd-B model without the relaxation term (infinite Weissenberg
number) and establish global-in-time  existence for small initial
data. Finally, Constantin \cite{Con05} studies a class of kinetic
models in the form of a Stokes equation coupled to a nonlinear
Fokker-Planck equation, for which he proves global-in-time
existence. The extension to a two-dimensional Navier-Stokes system
coupled to a nonlinear Fokker-Planck equation appears in Constantin
et al. \cite{CFTZ07} and Constantin and Masmoudi \cite{CM07}. In both cases the analysis benefits from the
fact that the polymeric stress remains bounded by construction,
which as our analysis shows, is the key to well-posedness.

\section{The model}
The Oldroyd-B model describes a fluid in which polymer molecules are
suspended in a viscous incompressible solvent. The equations of
motion in the creeping flow regime are
\begin{equation}
\begin{gathered}
0 =  -\grad p + \nu_s \Delta\vel +  \div\str \\
\div\vel = 0 \\
\pd{\str}{t} + (\vel\cdot\grad)\str - (\gvel)^T\str - \str (\gvel) =
-\frac{1}{\lambda}\str +\frac{\nu_p}{\lambda} \Brk{\gvel + (\gvel)^T},
\end{gathered}
\label{eq:the_model}
\end{equation}
where $\vel$ is the velocity of the fluid, $p$ is the hydrostatic
pressure, $\str$ is the extra-stress tensor due to the polymer
molecules, $\nu_s$ is the solvent viscosity, $\nu_p$ is the
polymeric viscosity, and $\lambda$ is the elastic relaxation time.
The velocity gradient is defined with components $(\gvel)_{ij}=
\partial u_j/ \partial x_i$.

The first two equations in \eqref{eq:the_model} are a Stokes system
for an incompressible fluid, whereas the third equation is the
Maxwell constitutive equation  for the extra-stress \cite{BAH87}.
The flow is assumed to take place in the unbounded three-dimensional
space $\R^3$. Initial data need only be prescribed for the stress,
$\str(\x,0) = \str_0(\x)$.

The system \eqref{eq:the_model} can be turned into a closed equation
for $\str$, by solving the Stokes system, and  expressing the flow
field in terms of the stress field. Specifically, the induced
velocity field is given by
\begin{equation*}
u_j(\x) = \frac{1}{8\pi\nu_s} \int_{\R^3} M^{(0)}_{jl}(\y)
\partial_k\sigma_{kl}(\x-\y)\,d\y, 
\end{equation*}
where
\[
M^{(0)}_{jl}(\y) = -\frac{\delta_{jl}}{|y|} + \frac{y_jy_l}{|y|^3}
\]
is the Stokes kernel (see Galdi \cite{Gal94}, pp 189--195). Using
integration by parts, it can be rewritten as
\begin{equation}
u_j(\x) = \frac{1}{8\pi\nu_s} \int_{\R^3} M^{(1)}_{jkl}(\y) \sigma_{kl}(\x-\y)\,d\y,
\label{eq:solve_stokes_comp}
\end{equation}
where
\[
M^{(1)}_{jkl}(\y) = -\frac{y_j\delta_{kl}}{|y|^3} +  \frac{3y_jy_ky_l}{|y|^5}.
\]
Here and below we adopt the Einstein summation convention, whereby
repeated indexes imply a summation unless otherwise specified. In
tensor notation we write \eqref{eq:solve_stokes_comp} as
\begin{equation}
\vel(\x) = \frac{1}{8\pi\nu_s} \int_{\R^3} \bs{M}^{(1)}(\y) : \str(\x-\y)\,d\y,
\label{eq:solve_stokes}
\end{equation}
where for 2-tensors $\bs{a},\bs{b}$, the $:$ product  is defined by
$\bs{a}:\bs{b} = \tr(\bs{a}^T \bs{b})$.
We then rewrite the constitutive equation as a
closed evolution equation for the stress field,
\begin{equation}
\deriv{\str}{t} = \bs{F}(\str), \qquad \str(\cdot,0) = \str_0,
\label{eq:closed_str}
\end{equation}
where
\begin{equation}
\bs{F}(\str) = - (\vel\cdot\grad)\str + (\gvel)^T\str + \str (\gvel)
-\frac{1}{\lambda}\str +\frac{\nu_p}{\lambda} \Brk{\gvel + (\gvel)^T},
\label{eq:F}
\end{equation}
and $\vel$ is given by \eqref{eq:solve_stokes}. Equation
\eqref{eq:closed_str} is viewed as an evolution equation or an
ordinary differential equation (ODE) in an infinite-dimensional
function space. We observe that the solution $\str$ is a symmetric
tensor whenever $\str_0$ is symmetric.

For later use, we derive the linear relation between the stress
field $\str$ and the velocity gradient $\gvel$, obtained by
differentiating \eqref{eq:solve_stokes} and integrating by parts.
This yields a singular integral  (the integrand is a homogeneous
function of degree $-3$ that averages to zero on the unit sphere),
from which one has to extract the singular
part,
\[
\partial_i u_j(\x) = -\frac{1}{5\nu_s}\brk{ \sigma_{ij}(\x) - \frac{1}{3} \delta_{ij}\tr\sigma(\x)}
+ \frac{1}{8\pi\nu_s} \PV \int_{\R^3} M_{ijkl}^{(2)}(\y)\sigma_{kl}(\x-\y) \,d\y,
\]
where
\[
M^{(2)}_{ijkl}(\y) = \frac{\delta_{ij}\delta_{kl}}{|y|^3} -
3 \frac{y_iy_j\delta_{kl} + 2y_jy_l\delta_{ki}+ \delta_{ij}y_ky_l}{|y|^5}
+\frac{15y_iy_jy_ky_l}{|y|^7},
\]
and $\PV$ stands for the principal value of a singular integral.
In tensor notation,
\begin{equation}
\gvel(\x) = -\frac{1}{5\nu_s}\brk{\str(\x) - \frac{\bs{I}}{3} \tr\str(\x)}  +
\frac{1}{8\pi\nu_s} \PV\int_{\R^3} \bs{M}^{(2)}(\y) : \str(\x-\y)\,d\y.
\label{eq:gvel}
\end{equation}

\section{Local-in-time existence}
\label{sec:local}
In this section we prove the local existence and uniqueness of
solutions to the  Oldroyd-B model \eqref{eq:the_model}. The proof is
based on energy  methods, and closely follows the existence proof
for the Euler and Navier-Stokes equations in Majda and Bertozzi
\cite{MB02}.   Differences between the two cases are highlighted
along the treatment.

We denote by $\Hm$ the Sobolev spaces of scalar, vector and tensor
fields in $\R^3$. The corresponding norms are denoted by
$\NormH{\cdot}{m}$ and are  defined  by
\[
\begin{aligned}
\NormH{f}{m} &= \brk{\sum_{|\alpha|\le m} \NormH{D^\alpha f}{0}^2}^{1/2},
\end{aligned}
\]
where here $f$ denotes either a scalar, a vector or a tensor field
and $\alpha=(\alpha_1,\alpha_2,\alpha_3)$ is a multi-index of
derivatives; the norm $\NormH{\cdot}{0}$ is the $\Ltwo$-norm. For
example, if $f = \bs{f}$ is a tensor field with components $f_{ij}$,
then
\[
\NormH{D^\alpha \bs{f}}{0}^2 =
\sum_{i,j=1}^3 \int_{\R^3} \Brk{\frac{\partial^{|\alpha|}}{
\partial x_1^{\alpha_1} \partial x_2^{\alpha_2}\partial x_3^{\alpha_3}} f_{ij}(\x) }^2\,d\x.
\]
We denote by $\NormLq{\cdot}{q}$ and $\NormLinf{\cdot}$  the
$L^q(\R^3)$ and $\Linf$ norms, respectively. Throughout this section
and the next one, we use the symbols $C$, $K$,  to
denote either positive constants, or, depending on the context,
bounded functions of their arguments.

Our local-in-time existence theorem is:

\begin{thm}
\label{th:local} Let $\str_0\in \Hm$ for $m>5/2$. Then there exists
a time $T>0$ depending on $\NormH{\str_0}{m}$ only, so that  the
Oldroyd-B system \eqref{eq:the_model}, or equivalently the Hilbert
space-valued ODE \eqref{eq:closed_str} has a solution $\str$ in the
class
\begin{equation}
 \str\in C([0,T];\Hm)\cap C^1([0,T];H^{m-1}(\R^3)).
\end{equation}
\end{thm}
Observe that system \eqref{eq:the_model} is time-reversible, hence Theorem~\ref{th:local} is also valid backward in time.

Since the proof is long and technical, we describe here its outline,
and prove each step in a separate subsection. Standard definitions
and inequalities are grouped in  Appendix~\ref{app:ineq}.

\paragraph{Subsection~\ref{subsec:mol}}
We start by constructing smooth approximations to $\str$.
We consider a mollified version of \eqref{eq:closed_str},
\begin{equation}
\deriv{\stre}{t} = \Fe(\stre), \qquad \stre(\cdot,0) = \str_0,
\label{eq:mol_intro}
\end{equation}
where $\e>0$ is the mollification parameter,
\[
\begin{split}
\Fe(\stre) &= -\Je [\vele\cdot\grad(\Je\stre)] + ( \grad \vele)^T\stre + \stre \grad  (\vele) \\
&-\frac{1}{\lambda}\stre +\frac{\nu_p}{\lambda} \Brk{\grad\vele +(\grad \vele )^T},
\end{split}
\]
$\vele$ is given by the integral \eqref{eq:solve_stokes} with $\str$ replaced by $\stre$,
and
the mollification operator $\Je$ is defined by \eqref{eq:Je}.  We then prove

\begin{prop}
\label{prop:mol}
Let $\str_0\in \Hm$ for $m> 3/2$. Then there exists a time $T_\e>0$ depending
on $\NormH{\str_0}{m}$ only, so that  \eqref{eq:mol_intro} has a unique solution
\[
\stre\in C^1([0,T_\e);\Hm).
\]
\end{prop}

The main reason for introducing the mollified equation is to enable us to use classical theory of evolution equations (ODEs) in Banach spaces to prove short time existence of unique solutions. Most importantly, the solutions of the mollified equation are regular enough, which is enabling us to use classical tools for deriving estimates without a need for additional justification.

Finally, we also observe that by the uniqueness of the solution of \eqref{eq:mol_intro}, the tensor $\stre$ is symmetric whenever $\str_0$ is symmetric. We will be using this fact later in our estimates.

\paragraph{Subsection~\ref{subsec:common}}
Using the continuation theorem for autonomous ODEs, and an  a-priori estimate of the form
\begin{equation}
\sup_{0\le t\le T} \NormH{\stre}{m} \le K(\NormH{\str_0}{m},T),
\label{eq:unif_bound}
\end{equation}
where $T = T(\NormH{\str_0}{m})$ is independent of $\e$, we show that the family of  mollified solutions $\stre$ can be continued uniformly up to a common time $T$. Here we need a slightly higher degree of regularity:

\begin{prop}
\label{prop:common}
Let $\str_0\in \Hm$ for $m> 5/2$.  Then there exists a time $T=T(\NormH{\str_0}{m})>0$ independent of $\e$, such that the mollified equation \eqref{eq:mol_intro} has a unique solution $\stre\in C^1([0,T],\Hm)$, satisfying the uniform bound \eqref{eq:unif_bound}.
\end{prop}

\paragraph{Subsection~\ref{subsec:cauchy}}
We show that the family of mollified solutions, $\stre$, forms a Cauchy sequence in $C([0,T],\Ltwo)$, hence strongly converges to a function $\str\in C([0,T],\Ltwo)$.

\begin{prop}
\label{prop:cauchy}
Let $\str_0\in \Hm$ for $m> 5/2$. Then
the family of mollified solutions $\stre\in C^1([0,T],\Hm)$ forms, as $\e\to0$, a Cauchy sequence in $C([0,T],L^2(\R^3))$, hence converges to a function, which we denote by $\str$. Moreover, for every $0\le t\le T$, we have $\str(\cdot,t)\in\Hm$, and $\str$ satisfies the same bound
\[
\sup_{0\le t\le T}  \NormH{\str}{m} \le K(\NormH{\str_0}{m},T),
\]
as the family of mollified solutions.
\end{prop}

\paragraph{Subsection~\ref{subsec:highnorms}}
Using the technique of interpolation we show that $\stre$ strongly converges to $\str$ in all intermediate norms $C([0,T],H^{m'}(\R^3))$, $0<m'<m$. We then proceed to prove continuity of the limit in the highest norm, $\str\in C([0,T],\Hm)$.

\begin{prop}
\label{prop:highnorms}
Let $\str_0\in \Hm$ for $m>5/2$. Then the $C([0,T],H^{m'}(\R^3))$ limit $\str$ of $\stre$, for every $m'\in(0,m)$ is continuous in $\str\in C([0,T],\Hm)$.
\end{prop}

\paragraph{Subsection~\ref{subsec:ode}}
We finally show that $\str$, the limit of $\stre$, is a solution of \eqref{eq:closed_str} in the space
\[
\str\in C([0,T],\Hm)\cap C^1([0,T],H^{m-1}(\R^3)).
\]

\subsection{Local-in-time existence of mollified solutions}
\label{subsec:mol}
In this subsection we prove Proposition~\ref{prop:mol}. We
approximate the Oldroyd-B system \eqref{eq:the_model}, or
equivalently the Hilbert space-valued ODE (evolution equation) \eqref{eq:closed_str} by a
mollified equation for a smooth approximation $\stre$ of $\str$,
\begin{equation}
\deriv{\stre}{t} = \Fe(\stre), \qquad \stre(\cdot,0) = \str_0,
\label{eq:mol}
\end{equation}
where
\begin{equation}
\begin{split}
\Fe(\stre) &= -\Je [\vele\cdot\grad(\Je\stre)] + (\gvele)^T\stre + \stre(\gvele) \\
&-\frac{1}{\lambda}\stre +\frac{\nu_p}{\lambda} \Brk{\grad\vele +(\grad \vele )^T},
\end{split}
\label{eq:Fe}
\end{equation}
$\vele$ is given by \eqref{eq:solve_stokes} with $\str$ replaced by $\stre$, and
the mollification operator $\Je$ is defined by \eqref{eq:Je} in the
appendix. Comparing with \eqref{eq:F}, we note that mollification is
only used in the advection term. As will be shown,  the gradient of
$\vele$ has the same degree of regularity, with respect to the
$\Hm$-norms, as $\stre$, hence no additional mollification is
needed.

To prove that \eqref{eq:mol} has a local-in-time solution we
use Picard's theorem over Banach spaces. Specifically,
we work within the Banach space $\Hm$ with $m>3/2$.
Picard's theorem for functional evolution differential equations
states that if there exists an open subset $O\subset \Hm$ such that
\begin{enumerate}
\item $\Fe:O \to \Hm$.
\item $\Fe$ is locally Lipschitz continuous, i.e., for any $\str\in O$ there exists
 an open neighborhood of $\str$, $U\subset O$,  and a constant $L>0$ such that for every $\btau_1,\btau_2\in
U$
\[
\NormH{\Fe(\btau_1) - \Fe(\btau_2)}{m} \le L \NormH{\btau_1 - \btau_2}{m}
\]
\end{enumerate}
then there exists, for every $\str_0\in O$, a time $T_\e>0$ and a unique solution $\stre\in C^1([0,T_\e);O)$ of \eqref{eq:mol}.

Two properties that are being used extensively throughout this
section are: (i) the Calder\'on-Zygmund (CZ) inequality
\eqref{eq:CZ2} (see Appendix), from which it follows that $\stre\in\Hm$ implies
that the velocity gradient $\gvele$, and a-fortiori the velocity
$\vele$ itself, are in $\Hm$ as well,
\begin{equation}
\NormH{\vele}{m} \le \frac{c_m}{\nu_s} \NormH{\stre}{m}\,\,
,\qquad \NormH{\gvele}{m} \le \frac{\tilde{c}_m}{\nu_s} \NormH{\stre}{m}.
\label{eq:edriss_num}
\end{equation}
(ii) For $m>3/2$, $\Hm$ forms a Banach algebra, i.e.,
\[
\NormH{f g}{m} \le c \NormH{f}{m} \NormH{g}{m}.
\]

Combining these two properties with the smoothing properties
\eqref{eq:J2}--\eqref{eq:J3} of $\Je$ (see Appendix), it follows at once that $\Fe$
is a mapping $\Hm\to\Hm$. We set
\[
O = \BRK{\str\in\Hm:\,\,\NormH{\str}{m} < r},
\]
where $r$ is sufficiently large such that $\str_0\in O$. It remains
to show that there exists a positive constant $L = L(r)$, such that
$\Fe$ is Lipschitz continuous in $O$ with constant $L$.

To avoid lengthy expressions, we
split $\Fe$ into a sum of four terms,
\[
\Fe = \F_1 + \F_2 + \F_3 + \F_4,
\]
where
\[
\begin{aligned}
&\F_1(\stre) &=& - \Je[ \vele \cdot \grad (\Je \stre)]
\qquad
&\F_2(\stre) &=& (\gvele)^T \stre + \stre(\gvele) \\
&\F_3(\stre) &=& -\frac{1}{\lambda} \stre
\qquad
&\F_4(\stre) &=& \frac{\nu_p}{\lambda}\Brk{\gvele + (\gvele)^T},
\end{aligned}
\]
and show that each of these four terms is Lipschitz continuous.
That is,
let $\btau_1,\btau_2\in O$ and let $\vel_1,\vel_2$ be their corresponding velocity fields satisfying the Stokes system \eqref{eq:solve_stokes}; we show that each of the $\F_j$ verifies a bound of the type
\[
\NormH{\F_j(\btau_1)-\F_j(\btau_2)}{m} \le L(\e,\NormH{\btau_1}{m},\NormH{\btau_2}{m})\NormH{\btau_1-\btau_2}{m},
\]
where $L$ is a monotonic function of its last two arguments, and hence it is bounded by $L(\e,r,r)$.

The Lipschitz continuity of the linear function $\F_3$ is trivial. The Lipschitz continuity of $\F_4$ follows from the CZ inequality \eqref{eq:CZ2},  which implies
\[
\NormH{\gvel_2 - \gvel_1}{m} \le \frac{C}{\nu_s}\, \NormH{\btau_2 - \btau_1}{m}.
\]

For the advection term $\F_1$  we have
\[
\begin{split}
\NormH{\F_1(\btau_2) - \F_1(\btau_1)}{m} &=
\NormH{\Je[ \vel_2 \cdot \grad (\Je \btau_2)] - \Je[ \vel_1 \cdot \grad (\Je \btau_1)]}{m} \\
&\le C\,\NormH{\vel_2 \cdot \grad (\Je \btau_2) - \vel_1 \cdot \grad (\Je \btau_1)}{m} \\
&\le C\,\NormH{\vel_2 \cdot \grad (\Je (\btau_2-\btau_1))}{m} + C\,\NormH{(\vel_2-\vel_1)\cdot \grad (\Je \btau_1)}{m},
\end{split}
\]
where in the passage from the first to the second line we used \eqref{eq:J2} with $k=0$, and in the passage from the second to the third line we added and subtracted equal terms and used the triangle inequality. Using the Sobolev calculus inequality \eqref{eq:calc_ineq},
\[
\begin{split}
\NormH{\F_1(\btau_2) - \F_1(\btau_1)}{m}  &\le C
\Big[
\Norm{\vel_2}_{L^\infty} \NormH{\grad (\Je (\btau_2-\btau_1)) }{m} \\
&+
\Norm{\grad (\Je (\btau_2-\btau_1))}_{L^\infty} \NormH{\vel_2}{m}
+
\Norm{\vel_2-\vel_1}_{L^\infty} \NormH{\grad (\Je \btau_1) }{m} \\
&+
\Norm{\grad (\Je \btau_1)}_{L^\infty} \NormH{\vel_2-\vel_1}{m}
\Big].
\end{split}
\]
Using then property \eqref{eq:J2} of the mollification operator and the Sobolev embedding \eqref{eq:SobEmb}, we obtain:
\[
\NormH{\F_1(\btau_2) - \F_1(\btau_1)}{m}  \le \frac{C}{\e} \brk{
\Norm{\vel_2}_{m} \NormH{\btau_2-\btau_1 }{m} +
 \NormH{\btau_1}{m}\NormH{\vel_2-\vel_1}{m}}.
\]
Finally, using CZ inequality,
\[
\NormH{\F_1(\btau_2) - \F_1(\btau_1)}{m}  \le C(\e,\NormH{\btau_1}{m},\NormH{\btau_2}{m}) \NormH{\btau_2-\btau_1}{m}.
\]
Remains the deformation term $\F_2$. Using the Banach algebra property \eqref{eq:BanachAlgebra},
\[
\begin{split}
\NormH{\F_2(\btau_2) - \F_2(\btau_1)}{m}  & =
\NormH{( \gvel_1)^T \btau_1 +  \btau_1 (\gvel_1) -
 (\gvel_2)^T \btau_2 -  \btau_2 (\gvel_2)}{m} \\
&\le 2\,\NormH{\btau_1 (\gvel_1) -  \btau_2 (\gvel_2)}{m} \\
&\le 2\NormH{\gvel_1}{m}\NormH{\btau_2-\btau_1}{m} + 2\NormH{\btau_2}{m}\NormH{\gvel_2-\gvel_1}{m}.
\end{split}
\]
One more application of the CZ inequality gives,
\[
\NormH{\F_2(\btau_2) - \F_2(\btau_1)}{m}   \le C\brk{
\NormH{\btau_1}{m}+\NormH{\btau_2}{m}}\NormH{\btau_2-\btau_1}{m}.
\]

This shows that  $\Fe$ is locally Lipschitz in $O$, hence local existence for the mollified equation \eqref{eq:mol}.

\subsection{Energy estimates for the mollified solutions}
\label{subsec:common}
In this subsection we prove Proposition~\ref{prop:common}, i.e., that the mollified solutions can be continued uniformly up to a time $T>0$ that does not depend on $\e$.
To do so,  we first obtain an a-priori estimate, whereby
if the solution $\stre$ exists up to time $T$, then
\[
\sup_{0\le t\le T} \NormH{\stre(\cdot,t)}{m} \le
C(T,\NormH{\str_0}{m}) \equiv K,
\]
where all the $\stre$ have the same initial condition. The existence of
the solution up to that time follows from the continuation theorem
for autonomous ODEs in Banach spaces. Taking for domain
\[
O = \BRK{\str\in\Hm:\,\,\NormH{\str}{m} \le K},
\]
the solution either exists up to time $T$, or leaves the set $O$ before that time, which we  would have ruled out by the above estimate.

To obtain an a-priori uniform bound on $\NormH{\stre(\cdot,t)}{m}$, we use
an energy estimate. Starting from the mollified equation
\eqref{eq:mol}, we take its $\alpha$-th derivative ($|\alpha|\le
m$), and then an inner product with $D^\alpha\stre$. This yields an
``energy"  equation,
\begin{equation}
\begin{split}
\half \deriv{}{t} \NormH{D^\alpha \stre}{0}^2 &+ \frac{1}{\lambda} \NormH{D^\alpha \stre}{0}^2 = -\brk{D^\alpha \stre, D^\alpha \Je[ \vele \cdot \grad (\Je \stre)]} \\
&+ 2 \brk{D^\alpha \stre, D^\alpha [\stre (\gvele)]}
+ \frac{2\nu_p}{\lambda} \brk{D^\alpha \stre, D^\alpha (\gvele)} \\
& \equiv I_1 + I_2 + I_3,
\end{split}
\label{eq:energy}
\end{equation}
where we used the symmetry of $\stre$ in the last two terms.

We observe that $\stre$ is as smooth as the initial data $\str_0\in\Hm$, and $\Je\stre\in C^\infty(\R^3)$, therefore the above estimates should be interpreted in the strong sense. In particular, the time derivative is classical and we do not require additional justification.

Since we need a bound that does not depend on $\e$, we cannot use the smoothing properties of $\Je$. On the other hand, we are not concerned by finite-time blow up as long as the time horizon is independent of $\e$.

We start with the advection term. Using the fact that $\Je$ commutes
with weak derivatives and is symmetric, we write
\[
-I_1 =  \brk{D^\alpha (\Je \stre), D^\alpha[ \vele \cdot \grad
(\Je \stre)]}.
\]
We then add and subtract from the second argument of the inner product
\[
\vele \cdot D^\alpha \grad(\Je\stre),
\]
whose inner product with $D^\alpha (\Je \stre)$ vanishes due to the
incompressibility of the flow, i.e., due to $\vele$ being
divergence-free. Using the Cauchy-Schwarz   inequality and
\eqref{eq:J2} for $k=0$ we obtain
\[
I_1 \le C \NormH{D^\alpha  \stre}{0}
\NormH{D^\alpha[ \vele \cdot \grad (\Je \stre)] - \vele \cdot  D^\alpha\grad(\Je\stre)}{0}.
\]
We then invoke the Sobolev calculus inequality \eqref{eq:calc_ineq2} to get
\[
I_1 \le C \NormH{D^\alpha  \stre}{0}
\Brk{\NormLinf{\grad\vele} \NormH{\grad(\Je\stre)}{m-1} +
\NormH{\vele}{m} \NormLinf{\grad(\Je\stre)}}.
\]
Using repeatedly the property  \eqref{eq:J2} of $\Je$, the Sobolev embedding theorem \eqref{eq:SobEmb}, and the CZ inequality, we get the bound
\begin{equation}
I_1 \le \frac{C}{\nu_s}  \NormH{\stre}{m}^3.
\label{eq:I1est}
\end{equation}
Note that in order to bound $\NormLinf{\grad\stre}$ by $\NormH{\stre}{m}$ we need $m>5/2$.

We turn to $I_2$, where we use the Cauchy-Schwarz inequality, the Banach algebra property of $\Hk$ for $k>3/2$ and the CZ inequality,
\begin{equation}
\begin{split}
I_2 &= 2 \brk{D^\alpha \stre, D^\alpha [\stre (\gvele)]}
\le 2 \NormH{\stre}{m} \NormH{\stre (\gvele)}{m} \\
&\le C \, \NormH{\stre}{m}^2 \NormH{\gvele}{m}
\le \frac{C}{\nu_s} \NormH{\stre}{m}^3.
\end{split}
\label{eq:I2est}
\end{equation}
Remains $I_3$, which we estimate using  the Cauchy-Schwarz inequality followed by the CZ inequality,
\begin{equation}
I_3 = \frac{2\nu_p}{\lambda}  \brk{D^\alpha \stre, D^\alpha \gvele} \le
C\frac{\nu_p}{\lambda}\,\NormH{D^\alpha \stre}{0} \NormH{D^\alpha \gvele}{0} \le
C \frac{\nu_p}{\lambda \nu_s}\NormH{\stre}{m}^2.
\label{eq:I3est}
\end{equation}

Combining the three estimates \eqref{eq:I1est}, \eqref{eq:I2est} and
\eqref{eq:I3est} we obtain an energy inequality
\[
\deriv{}{t} \NormH{\stre}{m} + \frac{1}{\lambda} \NormH{\stre}{m} \le
c_1 \frac{\nu_p}{\lambda\nu_s} \NormH{\stre}{m} + \frac{c_2}{\nu_s} \NormH{\stre}{m}^2,
\]
from which we conclude the existence of a time $T = T(\NormH{\str_0}{m})$, independent of $\e$, for which all the $\stre$ exist and have a common bound in $\Hm$,
\begin{equation}
\begin{split}
\sup_{0\le t\le T}\NormH{\stre}{m} &\le
\frac{c_3 e^{c_3 T} \NormH{\str_0}{m}}{c_3 + c_4(1-e^{c_3 T})\NormH{\str_0}{m}}
 \equiv K(\NormH{\str_0}{m},T),
\end{split}
\label{eq:thebound}
\end{equation}
where $c_3 = \lambda^{-1}(c_1 \nu_p/\nu_s-1)$, $c_4 = c_2/\nu_s$
and $T < \frac{1}{|c_3|} \log(1 + |c_3|/c_4 \NormH{\str_0}{m})$.
This concludes the proof of Proposition~\ref{prop:common}.


{\bfseries Remark:} Substituting back the dimensional parameters, our expression for the uniform existence time is
\[
T < \frac{\lambda}{|c_1 \nu_p/\nu_s-1|} \log\brk{1 + \frac{\nu_s|c_1 \nu_p/\nu_s-1|}{c_2\lambda \NormH{\str_0}{m}}}.
\]
First, it follows that this time remains finite in the $\lambda\to\infty$ limit (i.e., infinite Weissenberg number), in which case $T<\nu_s/c_2  \NormH{\str_0}{m}$. Second, this time is unbounded (hence, global-in-time existence follows) if $\nu_s/\nu_p>c_2$ and the initial data are sufficiently small, namely $\NormH{\str_0}{m} < |c_3|/c_4$.

\subsection{Convergence  of $\stre$ in $C([0,T],\Ltwo)$}
\label{subsec:cauchy}

We proceed to prove Proposition~\ref{prop:cauchy}, whereby  the sequence $\stre$ forms, as $\e\to0$,  a Cauchy sequence in the space $C([0,T];L^2(\R^3))$. Here, $T$ is the uniform existence time established in the previous subsection.
Specifically, we show that for $\stre,\stret\in \Hm$ solutions of \eqref{eq:mol} with the same initial condition $\str_0$, the following holds:
\[
\sup_{0\le t\le T}\NormH{\stre-\stret}{0}\le C(\NormH{\str_0}{m}\,,T)\max(\e,\e').
\]
Hence follows the existence of $\str\in C([0,T],\Ltwo)$, such that
\[
\lim_{\e\to0}\,\,\sup_{0\le t\le T} \NormH{\stre(\cdot,t) - \str(\cdot,t)}{0} = 0.
\]

As in the previous subsection, we start with an energy equation, this time for the difference $\stre-\stret$:
\begin{equation}
\begin{split}
\half\deriv{}{t}\NormH{\stre -\stret}{0}^2 &+ \frac{1}{\lambda}\NormH{\stre -\stret}{0}^2 =
\\ &-\brk{\Je [\vele\cdot\grad(\Je\stre)] - \Jet [\velet\cdot\grad(\Jet\stret)],\stre-\stret} \\
&+ 2\brk{\stre\gvele  - \stret\gvelet, \stre-\stret} \\
&+ \frac{2\nu_p}{\lambda}\brk{\gvele-\gvelet,\stre-\stret} \equiv I_1 + I_2 + I_3.
\end{split}
\label{eq:en2}
\end{equation}

$I_2$ and $I_3$ are easily estimated by the same manipulations as in the previous subsection,
\begin{equation}
\begin{split}
I_2 &\le
2\NormH{\stre\gvele  - \stret\gvelet}{0} \NormH{\stre-\stret}{0} \\
&= 2\NormH{(\stre-\stret)\gvele  - \stret(\gvelet-\gvele)}{0} \NormH{\stre-\stret}{0} \\
&\le 2\Brk{\NormH{\stre-\stret}{0} \NormLinf{\gvele} + \NormLinf{\stret} \NormH{\gvelet-\gvele}{0} }\NormH{\stre-\stret}{0} \\
& \le C \brk{\NormLinf{\gvele} +\NormLinf{\stret}}\NormH{\stre-\stret}{0}^2\\
&\le C
(\NormH{\stre}{m}+\NormH{\stret}{m})\NormH{\stre-\stret}{0}^2,
\end{split}
\label{eq:I2}
\end{equation}
and similarly,
\begin{equation}
I_3  \le
 \frac{2\nu_p}{\lambda} \NormH{\gvele-\gvelet}{0} \NormH{\stre -\stret}{0}\le C \NormH{\stre -\stret}{0}^2.
\label{eq:I3}
\end{equation}
Remains the advection term $I_1$, which we first split as follows,
\[
\begin{split}
-I_1 &=
\brk{(\Je-\Jet)[\vele\cdot\grad(\Je\stre)],\stre-\stret
}
+ \brk{\Jet[(\vele-\velet)\cdot\grad(\Je\stre)],\stre-\stret } \\
&+ \brk{\Jet[\velet\cdot\grad(\Je-\Jet)\stre],\stre-\stret }
+ \brk{\Jet[\vele\cdot\grad\Jet(\stre-\stret)],\stre-\stret }.
\end{split}
\]
The last term vanishes because $\vele$ is divergence-free. For the first three terms we use the Cauchy-Schwarz inequality, obtaining thus
\begin{equation}
\begin{split}
\frac{|I_1|}{\NormH{\stre-\stret}{0}} &\le
\NormH{(\Je-\Jet)[\vele\cdot\grad(\Je\stre)]}{0} +
\NormH{\Jet[(\vele-\velet)\cdot\grad\Je\stre]}{0} \\
&+ \NormH{\Jet[\velet\cdot\grad(\Je-\Jet)\stre]}{0}
\equiv A_1 + A_2 + A_3.
\end{split}
\label{eq:I1}
\end{equation}

By \eqref{eq:J2} the outer $\Jet$ can be replaced in $A_2,A_3$
by a constant prefactor. $A_2$ is easily estimated by
\begin{equation}
A_2 \le C \NormH{\vele-\velet}{0} \NormLinf{\grad(\Je\stre)} \le
C \NormH{\stre-\stret}{0} \NormH{\stre}{m}\,\,,
\label{eq:A3}
\end{equation}
where we have used the CZ inequality and the Sobolev embedding theorem, with $m>5/2$ (see Appendix).

Note that the two remaining terms have the factor $(\Je-\Jet)$, which is ``small" in the following sense. By \eqref{eq:J1} follows that
\[
\NormH{(\Je-\Jet)f}{0} \le \NormH{(\Je-I)f}{0} + \NormH{(\Jet-I)f}{0} \le C
\NormH{f}{1}\, \max(\e,\e').
\]
Thus, $A_1$ can be estimated by
\begin{equation}
A_1 \le C \NormH{\vele\cdot\grad(\Je\stre)}{1}\, \max(\e,\e')
 \le C \NormH{\stre}{m}^2\, \max(\e,\e'),
\label{eq:A1}
\end{equation}
where the last inequality follows from the very rough estimate of
the $H^1$-norm by the $H^{m-1}$ norm, and the
Banach algebra property of $\Hk$ for $k>3/2$.

$A_3$ verifies an estimate similar to $A_1$. Gathering the expressions for $A_1$, $A_2$, $A_3$, $I_2$, $I_3$, and substituting into the energy equation \eqref{eq:en2},
\begin{equation}
\half\deriv{}{t}\NormH{\stre -\stret}{0}^2 \le C\NormH{\stre}{m}^2
\max(\e,\e')\NormH{\stre-\stret}{0} +C
(\NormH{\stret}{m}+\NormH{\stre}{m})\NormH{\stre-\stret}{0}^2,
\label{eq:before3.18}
\end{equation}
We now use the uniform bound \eqref{eq:thebound}  to obtain
\[
\deriv{}{t}\NormH{\stre -\stret}{0} \le
C(\NormH{\str_0}{m})\Brk{ \max(\e,\e')+\NormH{\stre-\stret}{0}},
\]
which upon integrating yields
\begin{equation}
\sup_{0\le t \le T} \NormH{\stre-\stret}{0}\le e^{C(\NormH{\str_0}{m}) T}\max(\e,\e'),
\label{eq:cauchy}
\end{equation}
and we used here the fact that $\stre$ and $\stret$ satisfy the same initial conditions.
Therefore, $\stre$ is  a Cauchy sequence in the Banach space $C([0,T];\Ltwo)$ and hence it has a limit $\str\in C([0,T];\Ltwo)$.
In particular, \eqref{eq:cauchy} implies that for the
limit $\str$ we have
\[
\sup_{0\le t \le T} \NormH{\stre-\str}{0}\le e^{C(\NormH{\str_0}{m})T}\e.
\]

The uniform boundedness \eqref{eq:thebound} of the $\stre$ implies
by the Banach-Alaoglu theorem that for every $t\le T$ the sequence
$\stre(\cdot,t)$ has a subsequence that converges weakly in $\Hm$.
This limit must however coincide with the $L^2$ limit,
$\str(\cdot,t)$. Moreover, the Banach-Alaoglu theorem also implies
that
\[
\NormH{\str(\cdot,t)}{m} \le \liminf_{\e\to0} \NormH{\stre(\cdot,t)}{m}\,\,,
\]
thus  it follows from \eqref{eq:thebound} that for every $t\in[0, T]$,
\begin{equation}
\sup_{0\le t\le T} \NormH{\str(\cdot,t)}{m} \le K(\NormH{\str_0}{m},T).
\label{eq:str_bound}
\end{equation}
Note, however that we do not yet know that $\str$ is a continuous function from $[0,T]$  into $\Hm$. It is the task of the next subsection to show that $\str$ is in $C([0,T],\Hm)$.

\subsection{Continuity in $\Hm$}
\label{subsec:highnorms}
In this section we prove Proposition~\ref{prop:highnorms}, whereby $\str\in C([0,T],\Hm)$. We do it in several steps.

We start by showing that the mollified solutions $\stre$ converge to $\str$ in all ``intermediate" norms, $C([0,T],H^{m'}(\R^3))$, for all $0<m'<m$. For that we invoke the following interpolation lemma in Sobolev space:
\[
\NormH{\btau}{m'}\le C_m \NormH{\btau}{0}^{1-m'/m}\NormH{\btau}{m}^{m'/m},
\]
valid for all $\btau\in\Hm$ and $0\le m'\le m$.
Substituting $\btau=\stre-\str$,
\[
\NormH{\stre-\str}{m'}\le C_m \NormH{\stre-\str}{0}^{1-m'/m}\NormH{\stre-\str}{m}^{m'/m}.
\]
Using the uniform boundedness \eqref{eq:thebound} of $\stre$ and \eqref{eq:str_bound},
\[
\NormH{\stre-\str}{m}^{m'/m} \le \brk{\NormH{\stre}{m} +
\NormH{\str}{m}}^{m'/m} \le
[2K(\NormH{\str_0}{m},T)]^{m'/m},
\]
thus we obtain
\[
\sup_{0\le t\le T} \NormH{\stre(\cdot,t)-\str(\cdot,t)}{m'}\le
[2K(\NormH{\str_0}{m},T)]^{m'/m}\, \sup_{0\le t\le T}
\NormH{\stre(\cdot,t)-\str(\cdot,t)}{0}^{1-m'/m},
\]
i.e., uniform convergence $\stre\to\str$ in all intermediate norms.

To show that $\str$ is time-continuous in the highest norm, we first show that
$\str$ is time-continuous in the weak topology of $\Hm$. That is, we
claim that for every $\phi\in H^{-m}(\R^3)$,
\[
g(t) = \Average{\phi,\str(\cdot,t)}
\]
is continuous in time, where $\Average{\cdot,\cdot}$ is the dual pairing
between $\Hm$ and $H^{-m}(\R^3)$. Since $\str$ is time-continuous in
the strong topologies of all the intermediate norms, it follows that
$\Average{\phi,\str(\cdot,t)}$ is time-continuous for all $\phi\in
H^{-m'}(\R^3)$, but since the latter is dense in $H^{-m}(\R^3)$ and
$\str$ satisfies the uniform bound \eqref{eq:str_bound} in $H^m$,
the continuity of $g(t)$ follows.

As is well-known
continuity in the weak topology  of a Hilbert space supplemented by the continuity of the norm yields continuity in the strong topology. Thus, it remains to show that $\NormH{\str(\cdot,t)}{m}$ is time-continuous.

We start by showing continuity at the initial time $t=0$.
For $h>0$ we have
\[
\begin{split}
\NormH{\str(\cdot,h)-\str_0}{m}^2 &=
\NormH{\str(\cdot,h)}{m}^2 - \NormH{\str_0}{m}^2
- 2\brk{\str(\cdot,h)-\str_0, \str_0}_m
\end{split}
\]
where $(\cdot,\cdot)_m$ denotes the inner-product in $\Hm$. As
$h\to0^+$, the last term vanishes by the time-continuity  of $\str$
in the weak topology in $\Hm$. This yields,
\begin{equation}
\NormH{\str_0}{m}\le \liminf_{h\searrow0}\NormH{\str(\cdot,h)}{m}.
\label{eq:cont_zero}
\end{equation}

To obtain the reverse inequality we observe that by following
similar steps as in the above proof on can obtain the following
modification of \eqref{eq:str_bound},
\[
\sup_{0\le t\le \tau} \NormH{\str(\cdot,t)}{m} \le K(\NormH{\str_0}{m},\tau),
\]
for all $\tau\in(0,T]$, where $K$ is given by \eqref{eq:str_bound}.
Therefore, by taking $\tau=h$ and assuming $h$ is small enough, the above inequality and \eqref{eq:str_bound}  yields,
\[
\NormH{\str(\cdot,h)}{m}\le \NormH{\str_0}{m} + C h
\,\NormH{\str_0}{m}\brk{ 1 + \NormH{\str_0}{m}}.
\]
Since $\NormH{\str(\cdot,h)}{m}$ is bounded we may let $h\to0^+$, obtaining
\[
\limsup_{h\searrow\, 0}\NormH{\str(\cdot,h)}{m}\le
\NormH{\str_0}{m},
\]
which together with \eqref{eq:cont_zero} implies right-continuity at $t=0$; left-continuity at $t=0$
follows from the fact that the
Oldroyd-B system can be time reversed (unlike parabolic equations
such as the viscous
Navier-Stokes equations).

It remains to show that $\str$ is time continuous at any arbitrary
time $s\in[0,T]$. We use the fact that $\str(\cdot,s)\in\Hm$ to
construct a new set of mollified solutions with initial data
at $t=s$, $\tilde{\str}_\e(\cdot,s)=\str(\cdot,s)$. By the same line of
reasoning as before, these solutions converge in $C([0,T],\Ltwo)$ to
a solution $\tilde{\str}$ which belongs, at all times, to $\Hm$ and
is continuous at the initial time $t=s$. Let $\stre(\cdot,t)$ be as
before. Now, one can follow the same steps as in
Section~\ref{subsec:cauchy} to show that
\begin{eqnarray*}
\half\deriv{}{t} \NormH{\stre - \tilde{\str}_\e}{0}^2 \le C
(\NormH{\tilde{\str}_\e}{m}+\NormH{\stre}{m})\NormH{\stre}{m} \e
\NormH{\stre - \tilde{\str}_\e}{0} \\
+ C (\NormH{\tilde{\str}_\e}{m} +\NormH{\stre}{m})\NormH{\stre -
\tilde{\str}_\e}{0}^2
\end{eqnarray*}
(cf. \eqref{eq:before3.18}). Similar steps lead that around $t=s$
\[
\NormH{\tilde{\str}_\e}{m} \le K(\NormH{\stre(\cdot,s)}{m}).
\]
 Therefore, we have
\[
\deriv{}{t} \NormH{\stre - \tilde{\str}_\e}{0} \le C(\NormH{\stre(\cdot,s)}{m})\,
\brk{\e + \NormH{\stre - \tilde{\str}_\e}{0}}.
\]
Integrating, we obtain
\[
\NormH{\stre(\cdot,t) - \tilde{\str}_\e(\cdot,t)}{0} \le
e^{C(\NormH{\stre(\cdot,s)}{m})(t-s)} \NormH{\stre(\cdot,s) - \tilde{\str}_\e(\cdot,s)}{0} +
e^{C(\NormH{\stre(\cdot,s)}{m})(t-s)}\,\e.
\]
Now, as we let $\e\to0$, we know that $\stre(\cdot,s)\to\str(\cdot,s)$, and we know that $\tilde{\str}_\e(\cdot,s) = \str(\cdot,s)$, thus we conclude
\[
\NormH{\str(\cdot,t) - \tilde{\str}(\cdot,t)}{0} = 0.
\]
Since $\tilde{\str}$ is continuous at $t=s$ so is $\str$.

\subsection{$\str$ is a solution of \eqref{eq:closed_str}}
\label{subsec:ode}
Having shown that $\str\in C([0,T],\Hm)$, it remains to show that $\str$ is indeed a solution of \eqref{eq:closed_str}, and that
\[
\str\in C([0,T],\Hm) \cap C^1([0,T],H^{m-1}(\R^3)).
\]

To show that, we refer once more to the mollified solutions, whose evolution satifies the integral equation,
\[
\stre(\cdot,t) = \str_0 + \int_0^t \Fe(\stre(\cdot,s))\,ds.
\]
We now exploit the convergence of $\stre$ to $\str$ in all the intermediate norms. Specifically, we set $5/2<m'<m$, and claim that
\[
\stre\to\str \qquad \text{ in $C([0,T],\Hmp)$}
\]
implies that
\[
\begin{aligned}
& \Je [\vele\cdot\grad(\Je\stre)] \to  \vel\cdot\grad\str \\
& \stre (\gvele) \to  \str (\gvel) \\
&  \grad\vele \to \gvel
\end{aligned}
\]
in $C([0,T],H^{m'-1}(\R^3))$. The last two identities follow from
the CZ inequality \eqref{eq:CZ2} and the Banach algebra property
of $\Hk$ for $k>3/2$.
The convergence of the advection term follows from the same considerations, up to the loss of one order of regularity due to the gradient of $\stre$. Thus,
\[
\str(\cdot,t) = \str_0 + \int_0^t \F(\str(\cdot,s))\,ds,
\]
which proves that $\str\in C^1([0,T],H^{m-1}(\R^3))$, and satisfies the differential equation  \eqref{eq:closed_str}.

{\bfseries Remark:} Based on the previous remark, this also proves global-in-time existence for small initial data.

\section{A Beale-Kato-Majda breakdown condition}
\label{sec:BKM}
Having proved the local-in-time existence of solutions
to the Oldroyd-B equation \eqref{eq:the_model}, or its ODE representation \eqref{eq:closed_str}, we turn to the main purpose of this paper, which is the characterization of the breakdown of such solutions at finite time. By the continuation theorem for autonomous ODEs, if $T^*<\infty$ and $[0,T^*)$ is the maximal time of existence of the solution $\str$, then
\begin{equation}
\limsup_{t\nearrow T^*} \NormH{\str}{m} = \infty.
\label{eq:div1}
\end{equation}
Such a breakdown criterion is not informative enough, as it roughly says that ``a solution exists as long as it exists".
Our main theorem below provides a more concise breakdown condition, which is only associated with the stress itself, and does not involve any of its derivatives:

\begin{thm}
Let $\str$ be a local-in-time solution to \eqref{eq:the_model} in the class
\[
C([0,T);\Hm)\cap C^1([0,T);H^{m-1}(\R^3)),
\]
with $m\ge 3$.
Suppose that $[0,T^*)$ is the maximal time of existence, with $T^*<\infty$, then
\begin{equation}
\lim_{t \nearrow T^*}
\int_0^{t} \NormLinf{\str(\cdot,s)} \,ds=\infty.
\label{eq:div2}
\end{equation}
\label{th:BKM}
\end{thm}

The proof is similar in essence to the proof of the Beale-Kato-Majda theorem for the Euler  equations \cite{BKM84}. All is needed is an a-priori estimate of the form,
\begin{equation}
\NormH{\str(\cdot,t)}{m} \le C\brk{t,\int_0^t \NormLinf{\str(\cdot,s)} \,ds},
\label{eq:BKMineq}
\end{equation}
where $C$ is a continuous function of its arguments, hence \eqref{eq:div1} occurs only if \eqref{eq:div2} occurs. The estimate \eqref{eq:BKMineq} is derived in two steps, detailed in the next two subsections.

\subsection{A priori estimates for the $\Hm$ norm}
Our estimates rely on the following version of the
Gagliardo-Nirenberg inequality,
\[
\NormLq{D^k f}{q} \le C  \NormLinf{f}^{1-k/m}
\NormLq{D^m f}{{kq/m}}^{k/m},
\]
where $1\le k\le m$ and $1<p<\infty$, with $p=\frac{kq}{m}$. With
this, we prove the following lemma:

\begin{lem}
For $f,g,h\in\Hm$ the following triple-product inequality holds,
\begin{equation}
\int_{\R^3} |D^\alpha h|\, |D^\beta f|\, |D^{\alpha-\beta} g|\,d\x
\le C \NormH{h}{|\alpha|} \NormH{f}{|\alpha|}^{|\beta|/|\alpha|}
\NormH{g}{|\alpha|}^{|\alpha-\beta|/|\alpha|}
\NormLinf{f}^{1-|\beta|/|\alpha|}
\NormLinf{g}^{1-|\alpha-\beta|/|\alpha|} , \label{eq:GN2}
\end{equation}
where $\beta<\alpha$.
\end{lem}

\begin{proof}
We start with the triple product inequality
\[
\int_{\R^3} |D^\alpha h| |D^\beta f| |D^{\alpha-\beta} g|\,d\x \le
\NormH{D^\alpha h}{0} \NormLq{D^\beta f}{q} \NormLq{D^{\alpha-\beta} g}{p},
\]
where $1/p + 1/q = 1/2$. We then use twice the Gagliardo-Nirenberg inequality,
\[
\NormLq{D^\beta f}{q} \le C
\NormLinf{f}^{1-|\beta|/|\alpha|} \NormLq{D^\alpha
f}{{|\beta|q/|\alpha|}}^{|\beta|/|\alpha|}
\]
and
\[
\NormLq{D^{\alpha-\beta} g}{p} \le C
\NormLinf{g}^{1-|\alpha-\beta|/|\alpha|}  \NormLq{D^\alpha
g}{|\alpha-\beta|p/|\alpha|}^{|\alpha-\beta|/|\alpha|}.
\]
If we choose
\[
q = \frac{2|\alpha|}{|\beta|}
\Textand
p = \frac{2|\alpha|}{|\alpha-\beta|},
\]
then \eqref{eq:GN2} is obtained.
\end{proof}

We are going to make an extensive use of inequality \eqref{eq:GN2}
for the case where $f=\str$, $g=\gvel$ and $h=\str$ (more precisely,
$f,g,h$ are components of these  tensors). Then, combined with the
CZ inequality \eqref{eq:CZ2}, we get
\[
\brk{D^\alpha \str, (D^\beta \str)  [D^{\alpha-\beta} (\gvel)]} \le
C \NormH{\str}{|\alpha|}^2
\NormLinf{\str}^{1-|\beta|/|\alpha|}
\NormLinf{\gvel}^{1-|\alpha-\beta|/|\alpha|},
\]
which combined with Young's inequality finally gives,
\begin{equation}
\brk{D^\alpha \str, (D^\beta \str)  [D^{\alpha-\beta} (\gvel)]} \le
C \NormH{\str}{|\alpha|}^2  \brk{
\NormLinf{\str} +   \NormLinf{\gvel}}.
\label{eq:GN3}
\end{equation}

For every $|\alpha|\le m$
the $L^2$-norm of the $\alpha$-th derivative of $\str$ satisfies the energy equation
\begin{equation}
\begin{split}
\half \deriv{}{t} \NormH{D^\alpha\str}{0}^2 +
\frac{1}{\lambda} \NormH{D^\alpha\str}{0}^2  &=
-\brk{D^\alpha\str, D^\alpha(u_k\cdot\partial_k \str)} \\
&+ 2 \brk{D^\alpha\str, D^\alpha[(\str(\gvel)]} + \frac{2\nu_p}{\lambda} \brk{D^\alpha\str,D^\alpha(\gvel)}.
\end{split}
\label{eq:energy_BKM}
\end{equation}
The last term is easily estimated using the Cauchy-Schwarz inequality and the CZ
inequality \eqref{eq:CZ2},
\[
\brk{D^\alpha\str,D^\alpha(\gvel)} \le C \NormH{D^\alpha\str}{0}^2.
\]
The middle term can be written as
\[
\brk{D^\alpha\str, D^\alpha[(\str(\gvel)]} = \sum_{\beta\le\alpha}
\brk{D^\alpha\str, (D^\beta \str)(D^{\alpha-\beta}\gvel)},
\]
which is a finite sum of terms, each of which can be bounded  using \eqref{eq:GN3}.

Remains the advection term. Because $\vel$ is incompressible, the
term $u_k\partial_k(D^\alpha\str)$ vanishes, which means that $\vel$
is differentiated at least once, and we can use \eqref{eq:GN3} once
again. Thus, we obtain the inequality,
\[
\half \deriv{}{t} \NormH{D^\alpha\str}{0}^2 + \frac{1}{\lambda} \NormH{D^\alpha\str}{0}^2 \le
C\Brk{1 +
\NormLinf{\str} +   \NormLinf{\gvel}} \NormH{D^\alpha\str}{0}^2,
\]
and summing up over all $|\alpha|\le m$,
\[
\deriv{}{t} \NormH{\str}{m} + \frac{1}{\lambda} \NormH{\str}{m} \le
C\brk{1 + \NormLinf{\str} +   \NormLinf{\gvel}} \NormH{\str}{m}.
\]
A simple integration yields,
\begin{equation}
\NormH{\str}{m} \le \exp\Brk{C\int_0^t \brk{1 +
\NormLinf{\str} + \NormLinf{\gvel}}\,ds}\, \NormH{\str_0}{m}.
\label{eq:apriori}
\end{equation}

A comment:  the energy inequality \eqref{eq:energy_BKM} is only
formal since we have not shown that the $\Hm$ norm of $\str$ was
differentiable. To rectify this delicacy, one has to carry all
estimates with the mollified solutions $\stre$, which are
differentiable in all Sobolev spaces, and take the limit $\e\to0$,
only once we have obtained a final estimate for $\NormH{\stre}{m}$
in the integrated form .

We will need one more estimate. For all indices $i,j,k$ we have
\[
\pd{}{t} \partial_k\sigma_{ij} + \frac{1}{\lambda} \partial_k\sigma_{ij} =
-\partial_k \brk{u_l \partial_l \sigma_{ij}} +
\partial_k \brk{\sigma_{il} (\partial_l u_j) + (\partial_l u_i)\sigma_{lj}} +
\frac{\nu_p}{\lambda} \partial_k \brk{\partial_i u_j + \partial_j u_i}.
\]
Multiplying by $(\partial_k\sigma_{ij})^3$ (with summation over all indexes) and integrating over $\R^3$ we get
\[
\begin{split}
\frac{1}{4} \deriv{}{t} \NormLq{\grad\str}{4}^4 + \frac{1}{\lambda} \NormLq{\grad\str}{4}^4 &=
-\int_{\R^3} \brk{\partial_k\sigma_{ij}}^3
\partial_k \brk{u_l \partial_l \sigma_{ij}} \,d\x \\
& \hspace{-1cm}+ 2 \int_{\R^3} \brk{\partial_k\sigma_{ij}}^3 \partial_k \Brk{\sigma_{il} (\partial_l u_j)}\,d\x
+ \frac{2\nu_p}{\lambda} \int_{\R^3}  \brk{\partial_k\sigma_{ij}}^3 \partial_k \brk{\partial_i u_j}\,d\x.
\end{split}
\]
Using for the first two terms the triple product inequality, and the H\"older inequality for the third, we get
\[
\begin{split}
\frac{1}{4} \deriv{}{t} \NormLq{\grad\str}{4}^4 + \frac{1}{\lambda} \NormLq{\grad\str}{4}^4
&\le C \big[\NormLq{\grad\str}{4}^4 \NormLinf{\gvel} + \NormLinf{\str} \NormLq{\grad\gvel}{4}
\NormLq{\grad\str}{4}^3 \\
&+
 \NormLq{\grad\str}{4}^3 \NormLq{\grad\grad\vel}{4} \big],
\end{split}
\]
and after applying once again the CZ inequality \eqref{eq:CZ2},
\[
 \deriv{}{t} \NormLq{\grad\str}{4}
\le C\Brk{1+ \NormLinf{\gvel} + \NormLinf{\str}} \NormLq{\grad\str}{4},
\]
from which we get
\begin{equation}
\NormLq{\grad\str}{4} \le \exp\Brk{C\int_0^t \brk{1
+ \NormLinf{\str} +   \NormLinf{\gvel}}\,ds}\,
\NormLq{\grad\str_0}{4}. \label{eq:L4}
\end{equation}

\subsection{$L^\infty$ estimate of $\gvel$}
So far we have shown in \eqref{eq:apriori} that if
the solution $\str$ breaks down at time $T^*$ then
\[
\lim_{t\nearrow T^*} \int_0^t \brk{1 + \NormLinf{\str} +   \NormLinf{\gvel}}\,ds = \infty.
\]
To complete the proof of Theorem~\ref{th:BKM} it is sufficient to show that
\[
\int_0^t \brk{1 + \NormLinf{\str} +   \NormLinf{\gvel}}\,ds \le C\brk{t,
\int_0^t \NormLinf{\str} \,ds}.
\]
Thus, we need to estimate $\NormLinf{\gvel}$ in terms of $\NormLinf{\str}$. Note that the CZ inequality provides a bound for $\NormLq{\gvel}{p}$ in terms of $\NormLq{\str}{p}$ for all $p$ (such a bound exists within any of the Sobolev $W^{m,p}$ norms), but the prefactor  is linear in $p$, hence $p$ cannot be taken to be infinite. Instead, one has to perform a more delicate analysis.

Consider the integral relation \eqref{eq:gvel} between $\str$ and $\gvel$. We  split the domain of integration into an ``outer domain" $|\y|>R$, a ``middle annulus" $\e<|\y|<R$, and  an ``inner disc", $|\y|<\e$, namely,
\[
\gvel = -\frac{1}{5\nu_s}\brk{\str - \frac{\bs{I}}{3} \tr\str} + \frac{1}{8\pi\nu_s}
\brk{I_1 + I_2 + I_3},
\]
where
\[
\begin{aligned}
I_1(\x) &= \int_{R<|\y|} \bs{M}^{(2)}(\y): \str(\x-\y)\,d\y \\
I_2(\x) &= \int_{\e<|\y|<R} \bs{M}^{(2)}(\y): \str(\x-\y)\,d\y \\
I_3(\x) &= \PV\int_{|\y|<\e} \bs{M}^{(2)}(\y): \str(\x-\y)\,d\y.
\end{aligned}
\]
Recall that $\bs{M}^{(2)}$ is homogeneous of degree $-3$ and averages to zero on the unit sphere.

The ``outer" integral is estimated using the Cauchy-Schwarz inequality,
\[
|I_1(\x)|  \le  \frac{C}{R^{3/2}} \NormH{\str}{0}.
\]
The ``middle" integral is estimated by taking out the infinity norm of the stress
\[
\Abs{I_2(\x)} \le  C \NormLinf{\str} \log \frac{R}{\e},
\]
For the ``inner" integral, we exploit the fact that $\bs{M}^{(2)}$ averages to zero on the unit sphere to  subtract
\[
\PV \int_{|\y|<\e}  \bs{M}^{(2)}(\y): \str(\x)\,d\y = 0,
\]
so that,
\[
I_3 = \PV\int_{|\y|<\e}\bs{M}^{(2)}(\y): \Brk{\str(\x-\y)-\str(\x)}\,d\y.
\]
By the mean-value theorem
(this assumes that $\str\in C^1(\R^3)$, which is the case since
 $m>5/2$), we obtain
\[
\begin{split}
|I_3(\x) &\le \int_{|\y|<\e} \Abs{\bs{M}^{(2)}(\y):
\Brk{(\y\cdot\grad)\str(\bs{\xi})}}\,d\y \le C
\brk{\int_{|\y|<\e}|{M}^{(2)}(\y) \y|^p}^{1/p}
\NormLq{\grad\str}{q},
\end{split}
\]
where we used H\"older's inequality, and $\bs{\xi}$ represents an intermediate point. Since $\bs{M}^{(2)}$ is homogeneous of degree $-3$, this $L^p$-norm is finite provided that $2-2p>-1$, i.e., $p<3/2$,  and consequently $q>3$. Setting $q=4$ and
combining all three contributions, we get
\[
\Abs{\PV\int_{\R^3} \bs{M}^{(2)}(\y): \str(\x-\y)\,d\y} \le
C\,\brk{\frac{1}{R^{3/2}}  \NormH{\str}{0} +
\NormLq{\grad\str}{4} \e^{1/4} + \NormLinf{\str}
\log \frac{R}{\e}}.
\]
It remains to choose $R,\e$ such to minimize the bound. Taking
\[
R = \brk{\frac{\frac32 \NormH{\str}{0}}{ \NormLinf{\str}}}^{2/3}
\quad\text{ and }\quad \e =
\brk{\frac{4\NormLinf{\str}}{\NormLq{\grad\str}{4}}}^{1/4}
\]
we finally obtain
\begin{equation}
\NormLinf{\gvel} \le
C\, \NormLinf{\str}\brk{1 + \log_+ \NormH{\str}{0} + \log_+ \NormLq{\grad\str}{4} },
\label{eq:gvel_ineq}
\end{equation}
where $\log_+ x = \max(\log x,0)$.

The estimate \eqref{eq:apriori} with $m=0$ gives,
\[
\NormH{\str}{0} \le \exp\Brk{C\int_0^t \brk{1 +
\NormLinf{\str} + \NormLinf{\gvel}}\,ds}\, \NormH{\str_0}{0}.
\]
and similarly, from \eqref{eq:L4},
\[
\NormLq{\grad\str}{4} \le
 \exp\Brk{C\int_0^t \brk{1 + \NormLinf{\str} +   \NormLinf{\gvel}}\,ds}\,
\NormLq{\grad\str_0}{4}.
\]
Substituting into \eqref{eq:gvel_ineq} we get
\[
\NormLinf{\gvel} \le C \NormLinf{\str}\Brk{ 1 +
\int_0^t \brk{1 + \NormLinf{\str} +   \NormLinf{\gvel}}\,ds}.
\]
If we define
\[
M(t) = \int_0^t \Brk{1 + \NormLinf{\str} +   \NormLinf{\gvel}}\,ds
\Textand
N(t) = \int_0^t \NormLinf{\str}\,ds,
\]
then we have an integral inequality of the form
\[
M'(t) \le C N'(t)[1 + M(t)],
\]
which we readily integrate,
\begin{equation}
\int_0^t \Brk{1 + \NormLinf{\str} +   \NormLinf{\gvel}}\,ds \le
\exp\brk{C \int_0^t  \NormLinf{\str}\,ds}.
\label{eq:4.10}
\end{equation}
Combining with \eqref{eq:apriori} we have proved Theorem~\ref{th:BKM}.

Note that \eqref{eq:apriori} together with \eqref{eq:4.10} yield a doubly-exponential bound,
\[
\NormH{\str}{m} \le  \exp\Brk{C\exp\brk{C \int_0^t
\NormLinf{\str}\,ds}}\,\NormH{\str_0}{m}.
\]
In this context, it is noteworthy that a similar doubly-exponential bound was derived in BKM for the Euler equation. Ponce in \cite{Pon85} derives a singly-exponential bound using as control parameter the $L^\infty$ norm of the deformation tensor, rather than the vorticity as in the original BKM paper.

\section{Discussion}
In this paper we derived a breakdown condition for solutions of the
Oldroyd-B model in $\R^3$ in the limit of zero Reynolds number. This
condition is analogous to the breakdown condition of
Beale-Kato-Majda for Newtonian fluids. It is noteworthy that the
elastic relaxation time, $\lambda$ (which in non-dimensional
formulations is the Weissenberg number) plays no role in our
analysis. In fact, nothing changes if we set $\lambda=\infty$, or,
alternatively, set $\lambda=\infty$ but retain the ratio
$\nu_p\lambda$ constant (which corresponds to the Kelvin-Voigt model
for a viscoelastic solid).

The main implication of our result is that the efforts toward a
global-in-time well-posedness theory should focus on the control of
the $L^\infty$ norm of the stress.

In a more physically realistic setting, one should
consider the same problem in a bounded domain $\Omega$. Let us
assume for simplicity homogeneous Dirichlet boundary conditions for
$\vel$, and a sufficiently smooth boundary $\partial\Omega$. The Caldero\'n-Zygmund
inequality holds in this case (see e.g. \cite{Yud63,Gal94}), so that local-in-time
existence can be proved as for the infinite domain.
Differences arise in
the proof of the BKM criterion, where we need an $L^{\infty}$
estimate for $\gvel$. For a bounded domain, the integral
representation \eqref{eq:gvel} can be rewritten with the same kernel
integrated over $\Omega$, plus a boundary-term contribution,
exactly as in the Poisson representation formula (see \cite{Gal94}).
Another alternative is to use the following Green
representation formula,
\[
\gvel(\x) = \int_{\Omega}
\mathcal{G}_{\Omega}(\y)\cdot\div\sigma(\x-\y)d\y
\]
where $\mathcal{G}_{\Omega}$ is a singular kernel that depends on
$\Omega$. The existence of $\mathcal{G}_\Omega$ along with
pointwise estimates has to be proved using elliptic
regularity theory (see Ferrari \cite{Fer93} for details).

\appendix
\section{Some inequalities}
\label{app:ineq}
In this appendix we list a number of inequalities used repeatedly in Sections~\ref{sec:local} and \ref{sec:BKM}. We recall that $\Hm$ denotes the Sobolev spaces of scalar, vector and multi-dimensional tensor fields, with the corresponding norm $\NormH{\cdot}{m}$. The $L^2$-norm, which coincides with the $H^0$-norm is denoted by $\NormH{\cdot}{0}$. The $L^\infty$ norm is denoted by $\NormLinf{\cdot}$. Weak derivatives are denoted by $D^\alpha$, where $\alpha = (\alpha_1,\alpha_2,\alpha_3)$ is a multi-index.

\paragraph{Mollifiers}
For $f\in
L^p(\R^3)$, $1\le p\le \infty$,
the mollification operator, $\Je$, is defined  by
\begin{equation}
(\Je f)(\x) = \frac{1}{\e^3} \int_{\R^3} \phi\brk{\frac{\x-\y}{\e}} f(\y)\,d\y,
\label{eq:Je}
\end{equation}
where $\phi:\R^3\to\R$ is a radially symmetric, positive, compactly supported $C^\infty$ function, satisfying $\int_{\R^3}\phi(\x)d\x=1$. The important properties satisfied by $\Je$  are
\begin{itemize}
\item $\Je$ commutes with (distributional) derivatives.
\item $\Je$ is symmetric with respect to the $\Ltwo$ inner product.
\item $\Je:\Hm\to\Hm\cap\Cinf(\R^3)$.
\item There exists a $C>0$ such that for every $f\in\Hm$,
\begin{equation}
\NormH{\Je f - f}{m-1} \le C \e \NormH{f}{m}.
\label{eq:J1}
\end{equation}
\item For every $f\in\Hm$ and $k\ge 0$,
\begin{equation}
\NormH{\Je f}{m+k} \le \frac{c_{mk}}{\e^k} \NormH{f}{m}.
\label{eq:J2}
\end{equation}
\item For every $f\in\Hm$ and multi-index $\alpha$,
\begin{equation}
\NormLinf{\Je D^\alpha f} \le \frac{c_{|\alpha|}}{\e^{3/2 + |\alpha|}} \NormH{f}{0}.
\label{eq:J3}
\end{equation}
\end{itemize}

\paragraph{Banach algebra property of Sobolev spaces}
For $m>3/2$ the Sobolev space  $\Hm$ is a Banach algebra, i.e.,
there exists a constant $C>0$ such that for all
$f,g\in\Hm$,
\begin{equation}
\NormH{f g}{m} \le C \NormH{f}{m} \NormH{g}{m}.
\label{eq:BanachAlgebra}
\end{equation}

\paragraph{Sobolev calculus inequalities}
For all $m\in\{0,1,2,..\}$ there exists a constant $C>0$ such that
for all $f,g\in\Linf\cap\Hm$,
\begin{equation}
\NormH{f g}{m} \le C\brk{\NormLinf{f} \NormH{g}{m} +  \NormH{f}{m} \NormLinf{g}}.
\label{eq:calc_ineq}
\end{equation}
For $m>3/2$ the Sobolev embedding $\Hm\subset\Linf$ implies that this inequality is valid for all $f,g\in\Hm$. Moreover, for $f\in\Hm$, $g\in H^{m-1}(\R^n)$, and $|\alpha|\le m$,
\begin{equation}
\NormH{D^\alpha(f g) - f D^\alpha g}{0} \le C
\Brk{\NormLinf{\grad f} \NormH{g}{m-1} + \NormH{f}{m} \NormLinf{g}}.
\label{eq:calc_ineq2}
\end{equation}
This inequality holds even if each of the two subtracted terms on the left hand side is not in $\Ltwo$ \cite{BKM84}.

\paragraph{Gagliardo-Nirenberg inequalities}
The classical Gagliardo-Nirenberg inequality is
\[
\NormLq{D^k f}{q} \le C  \NormLq{f}{r}^{1-\theta}  \NormLq{D^m f}{p}^{\theta},
\]
where $\theta=k/m\in(0,1)$, $1\le p,r \le \infty$, and
\[
\frac{1}{q} = \frac{\theta}{p} + \frac{1-\theta}{r}.
\]
For the  particular case $r=\infty$ we have
\begin{equation}
\NormLq{D^k f}{q} \le C  \NormLinf{f}^{1-k/m}  \NormLq{D^m f}{{kq/m}}^{k/m}\,\,,
\label{eq:GN}
\end{equation}
where $1<p<\infty$ and $1\le m\le k$.

\paragraph{Sobolev embedding theorems (\cite{Ada75} Ch. 5, \cite{Bre83} p. 168)}
Let $m\ge 1$ an integer and $1\le p<\infty$. Then:
\[
\begin{array}{lll}
\textrm{If} \,\,\, \frac{1}{p}-\frac{m}{n}>0\,\,\, & \,\,\, \textrm{then}\,\,\,
& W^{m,p}(\mathbb{R}^n)\subset L^q(\mathbb{R}^n) \,\,\, \textrm{with} \,\,\, \frac{1}{q} = \frac{1}{p}-\frac{m}{n}\\
\textrm{If} \,\,\, \frac{1}{p}-\frac{m}{n}=0\,\,\, & \,\,\, \textrm{then}\,\,\,
& W^{m,p}(\mathbb{R}^n)\subset L^q(\mathbb{R}^n) \,\,\, , \,\,\, \forall q\in [p,+\infty)\\
\textrm{If} \,\,\, \frac{1}{p}-\frac{m}{n}<0\,\,\, & \,\,\, \textrm{then}\,\,\,
& W^{m,p}(\mathbb{R}^n)\subset L^{\infty}(\mathbb{R}^n)
\end{array}
\]
In this paper we use extensively the third embedding for $p=2$ and $n=3$, i.e.
\[
\Hm\subset\Linf \qquad\text{ if }\qquad m>\frac{3}{2}.
\]
In fact, this is a continuous embedding, i.e.,
\begin{equation}
\NormLinf{f} \le C\,\NormH{f}{m}.
\label{eq:SobEmb}
\end{equation}

\paragraph{The Calder\'on-Zygmund inequality}
Let $K:\R^3\to\R$ be a homogeneous function of degree $-3$ that averages to zero on the unit sphere, and for $f\in L^q(\R^3)$, $q\ge2$,  define
\[
g(\x) = \PV\int_{\R^3} K(\y) f(\x-\y)\,d\y.
\]
Then $g\in L^q(\R^3)$ and
\begin{equation}
\NormLq{g}{q} \le C q\,\NormLq{f}{q}.
\label{eq:CZ1}
\end{equation}
Because the constant grows unbounded with $q$, this inequality does not carry to the $L^\infty$-norm;
see, for example,  Stein \cite{Ste70} for a proof. In particular, since the same kernel relates $D^\alpha f$ and $D^\alpha g$, it follows that $f\in\Hm$ implies $g\in\Hm$ with
\begin{equation}
\NormH{g}{m} \le C \,\NormH{f}{m}.
\label{eq:CZ2}
\end{equation}


{\bfseries Acknowledgments}
RK and CM were partially supported by the Israel Science Foundation.
The work of EST was supported in part by
the NSF grant no.~DMS-0504619, the ISF grant no.~120/6, and the BSF
grant no.~2004271.

\end{document}